\documentclass[a4paper,12pt]{article}

\usepackage{palatino}
\usepackage{mathpazo}

\usepackage{mathrsfs}
\usepackage{latexsym,amssymb,amsmath,amsthm,url}

\usepackage{float}
\usepackage{enumerate}

\theoremstyle{plain} 

\newtheorem{lemma}{Lemma}[section]
\newtheorem{thm}[lemma]{Theorem}
\newtheorem{prop}[lemma]{Proposition}
\newtheorem{cor}[lemma]{Corollary}

\theoremstyle{definition} 

\newtheorem{definition}[lemma]{Definition}

\theoremstyle{remark}
\newtheorem{remark}[lemma]{Remark}
 
\title{Linear growth for semigroups which are disjoint unions of finitely many copies of the free monogenic semigroup}

\begin{document}

\maketitle
\vskip .1in
\centerline{\large \bf Nabilah Abughazalah, Pavel Etingof}
\vskip .3in

\begin{abstract}
We show that every semigroup which is a finite disjoint union of copies of the free monogenic
semigroup (natural numbers under addition) has linear growth. 
This implies that the the corresponding semigroup algebra is a PI algebra. 
\bigskip

\noindent
\end{abstract}
\section{Introduction}
It is well known that a semigroup may sometimes be decomposed into a disjoint union of subsemigroups which is unlike the structures of classical algebra such as groups and rings. For instance, the Rees Theorem states that every completely simple semigroup is a Rees matrix semigroup over a group $G$, and is thus a disjoint union of copies of $G$, see \cite[Theorem 3.3.1]{howie95}; every Clifford semigroup is a strong semilattice of groups and as such it is a disjoint union of its maximal subgroups, see \cite[Theorem 3.3.1]{howie95}; every commutative semigroup is a semilattice of archimedean semigroups, see \cite[Theorem 3.3.1]{grillet95}.

\quad  If $S$ is a semigroup which can be decomposed into a disjoint union of subsemigroups, then it is natural to ask how the properties of $S$ depend on these  subsemigroups. For
example, if the subsemigroups are finitely generated, then so is $S$. Araújo et al. \cite {araujo01} consider the finite presentability of semigroups which are disjoint unions of finitely presented subsemigroups; Golubov \cite {golubov75} showed that a semigroup which is a disjoint union of residually finite subsemigroups is residually finite.

\quad For semigroups $S$ which are disjoint unions of finitely many copies of the free monogenic semigroup, the authors in \cite {abughazalah13} proved
that $S$ is finitely presented and residually finite; in \cite {abughazalah14} the authors proved that, up to isomorphism and anti-isomorphism, there are only two types of semigroups which are unions of two copies of the free monogenic
semigroup. Similarly, they showed that there are only nine types of semigroups which are unions of three copies of the free monogenic semigroup and provided finite presentations for semigroups of each of these types. 

\quad In this paper we continue investigating semigroups which are disjoint unions of finitely many copies of the free monogenic semigroup. Specifically, our main result is that every semigroup which is a finite disjoint union of copies of the free monogenic
semigroup has linear growth. This implies that the the corresponding semigroup algebra is a PI algebra. 
\\

{\bf Acknowledgments.}   The work of P.E. was partially supported by the NSF grant DMS-1000113. 
Both authors gratefully acknowledge the support of Ibn Khaldun Fellowship (Aramco). The first author is deeply grateful for the generous hospitality and excellent working conditions with Prof Pavel Etingof at MIT. The first author wants to thank Princess Nourah bint Abdulrahman University (PNU) for supporting the fellowship.

\section{The linear growth theorem}

Let $S$ be a semigroup which is a disjoint union of finitely many 
copies of the free monogenic semigroup:
$$S=\bigcup_{a\in A} N_a,$$
where $A$ is a finite set and $N_a=\langle a\rangle$ for $a\in A$.
Define the functions \linebreak $f: S\times A\times {\Bbb Z}_{>0}\to A$ and $n: S\times A\times {\Bbb Z}_{>0}\to {\Bbb Z}_{>0}$ 
by the formula 
$$
xy^i=f(x,y,i)^{n(x,y,i)}.
$$
For $y,z\in A, x\in S$, let $E(x,y,z)$ be the set of $i\in {\Bbb Z}_{>0}$ such that $f(x,y,i)=z$ (i.e., $xy^i$ is a power of $z$).  
\begin{definition}
We say that $z$ is $(x,y)- \operatorname{persistent}$ if $E(x,y,z)$
is infinite. 
\end{definition}

\begin{lemma}
\label{firstlem} (i) Let $z$ be $(x,y)$-persistent. 
Then $E(x,y,z)=i_0+E_0(x,y,z)$, where $i_0\in {\Bbb Z}_{>0}$
and $E_0(x,y,z)$ is a nonzero submonoid of ${\Bbb Z}_{\ge 0}$
(an additively closed subset containing $0$). 
Moreover, there exists a rational number $M_x(y,z)\ge 0$,   
such that for any $i\in E(x,y,z)$,  
$$
n(x,y,i)=n_0+M_x(y,z)(i-i_0).
$$

(ii) $M_x(y,z)$ is independent on $x$. 

(iii) The function $f(x,y,i)$ is periodic in $i$ starting from some place. 
\end{lemma} 

\begin{proof} 
First note that if $z^ty^q=z^s$ for some $t,q,s>0$
then for $N\ge t$ we have 
$z^Ny^q=z^{N+s-t}$, and hence 
\begin{equation}\label{eeq1}
z^Ny^{qm}=z^{N+(s-t)m}
\end{equation} 
for all $m>0$ and $N\ge mt$. 

Now let $i_0$ be the smallest number in $E(x,y,z)$, and $xy^{i_0}=z^{n_0}$. 
By our assumption, there exists $r>0$ such that 
$xy^{i_0+r}=z^n$; let us take the smallest such $r$. 
Thus $z^{n_0}y^r=z^n$. Let $M=M_x(y,z):=\frac{n-n_0}{r}$.
Then by \eqref{eeq1}, for any $p>0$ and $N\ge pn_0$, 
we have 
\begin{equation}\label{eeq2}
z^Ny^{pr}=z^{N+Mpr}.
\end{equation}

Suppose now that $z^ty^q=z^s$ for some $t,q,s>0$.
Then setting $m=r$ in \eqref{eeq1}, $p=q$ in \eqref{eeq2},
and comparing these equations, we get 
\begin{equation}\label{eeq1a}
s-t=Mq.
\end{equation}  

Now, let $E_0(x,y,z)=E(x,y,z)-i_0$. 
By \eqref{eeq1a}, for all $j\in E_0(x,y,z)$, we have  
\begin{equation}\label{eeq4}
z^{n_0}y^j=z^{Mj+n_0}
\end{equation} 
Since this holds for infinitely many $j$, we 
have $Mj+n_0>0$ for infinitely many $j$, hence 
$M\ge 0$. 

Let us show that $E_0(x,y,z)$ is closed under addition. 
Suppose $j,k\in E_0(x,y,z)$. Multiplying \eqref{eeq4} by $y^k$, 
we get 
$$
z^{n_0}y^{j+k}=z^{Mj+n_0}y^k=z^{Mj}z^{Mk+n_0}=z^{M(j+k)+n_0}.
$$
Thus, $j+k\in E_0(x,y,z)$. This proves (i). 

To prove (ii), let $x'\in S$ be another element 
such that $z$ is $(x',y)$-persistent. 
Then by \eqref{eeq2}, we have 
\begin{equation}\label{eeq3a}
z^Ny^{pr}=z^{N+Mpr},\ z^Ny^{pr'}=z^{N+M'pr'}
\end{equation} 
for all $p>0$ and sufficiently large $N$,
where $M'=M_{x'}(y,z)$. 
Thus, 
$$
z^Ny^{rr'}=z^{N+Mrr'}=z^{N+M'rr'}
$$
for large enough $N$, implying $M=M'$, as desired. 

Finally, note that it follows from (i) that 
if $f(x,y,i)=z$ for some $(x,y)$-persistent $z$, and if 
$j\in E_0(x,y,z)$, then $f(x,y,i+j)=z$. 
So statement (iii) follows from the fact that for sufficiently 
large $j$, $xy^j=z^m$ for an $(x,y)$-persistent $z$. 
Namely, the period $T$ for $f(x,y,j)$ can be taken to be 
any positive element of the intersection 
$$
\bigcap_{z\text{ is }(x,y)-\text{persistent}} E_0(x,y,z). 
$$
\end{proof}  

\begin{remark}
 Lemma 2.2 in \cite{abughazalah13} states that if $x^iy^j=x^k$ holds in $S$ for some $x,y\ \in A$ and $i,j,k\in\mathbb{N}$ then $i\leq k$. By this lemma and Lemma \ref{firstlem}, if the set $E(x,y,z)$ contains more than one element then it contains an arithmetic progression, and hence is infinite. 
\end{remark}

In view of Lemma \ref{firstlem}(ii), 
we will say that $z$ is $y$-persistent if there is $x\in S$ 
such that $z$ is $(x,y)$-persistent, and will denote $M_x(y,z)$ simply by $M(y,z)$.  
Notice that $y$ is obviously $y$-persistent, with $M(y,y)=1$.
  
 \begin{lemma}
 \label{lem2.3} Let $x,y,z\in A$. 
 If $z$ is $y-\operatorname{ persistent}$ and $y$ is $x-\operatorname{ persistent}$ then $z$ is $x-\operatorname{persistent}$ and $M(x,z)=M(x,y)M(y,z).$
 \end{lemma}
 
 \begin{proof} Suppose $z$ is $(v,y)$-persistent, and $y$ is 
$(u,x)$-persistent. 
Then, by Lemma \ref{firstlem}, 
$$
ux^{i_0+rm}=y^{n_0+M(x,y)rm},\quad vy^{i_0'+r'm'}=z^{n_0'+M(y,z)r'm'}
$$
for appropriate $i_0,i_0',n_0,n_0',r,r'\in \Bbb Z_{>0}$ and any $m,m'\in 
\Bbb Z_{\ge 0}$. Let $m=sr'$ with $s\in \Bbb Z_{\ge 0}$, and 
$\ell\ge 0$ be such that $\ell+n_0-i_0'=pr'$ for some $p\in \Bbb Z_{\ge 0}$. 
Then 
$$
(vy^\ell u)x^{i_0+rm}=vy^{\ell+n_0+M(x,y)rm}=vy^{i_0'+r'(p+M(x,y)rs)}=
z^{n_0'+M(y,z)r'(p+M(x,y)rs)}.
$$
Thus, 
$$
(vy^\ell u)x^{i_0+rr's}=z^{n_0'+M(y,z)pr'+M(x,y)M(y,z)rr's}.
$$
By Lemma \ref{firstlem}, this means that $z$ is 
$(vy^\ell u,x)$-persistent, and $M(x,z)=M(x,y)M(y,z)$, as desired. 
\end{proof} 
 
 \begin{lemma} 
 \label{lem2.4} Let $A$ be a finite set, $P\subset A\times A$ be a 
reflexive transitive relation, and $M: P\to \Bbb Z_{\ge 0}$ 
be a nonnegative integer function such that $M(y,y)=1$, and 
$M(x,z)=M(x,y)M(y,z)$ whenever $(x,y),(y,z)\in P$. 
Then there exists a positive integer 
function $\operatorname{d}$ on $A$ such that 
$$\operatorname{d}(y)\geq \operatorname{d}(z)M(y,z)$$
 whenever $(y,z)\in P$. 
 \end{lemma}
 \begin{proof}
 We say that $z$ is reachable from $y$ (denote by $z\geq y$) if
$(y,z)\in P$. We define a relation $\rho$ on $A$ as follows:
 \[\rho=\{(y,z)\subseteq A\times A\ \text{such that}\  y\ \text{is reachable from }\ z\ \text{and}\ z\ \text{is reachable from}\  y\}.\]
 It is clear that the relation $\rho$ is reflexive, symmetric, and 
transitive, so it is an equivalence relation, 
splitting $A$ into equivalence classes $A_1,A_2,\dots,A_n.$ 
Now we show that $\operatorname{d}$ exists if $n=1$, which means that $A$ is a single equivalence class. In this case, $M(x,y)$ 
is defined for all $x,y$, and $M(x,y)M(y,x)=M(y,y)=1$, so $M(x,y)>0$. 
Let us pick $a\in A$ and define $\operatorname{d}$ as follows:
 \[\operatorname{d}(b)=M(b,a)\ \forall \ b\in A.\] We can multiply this function by an integer to make it integer-valued. 
By our assumption, this function satisfies the required 
condition (in fact we have a stronger condition $\operatorname{d}(y)= \operatorname{d}(z)M(y,z)$).

 Now consider the case $n>1$. We say that an equivalence class $C$ is a sink if each element reachable from $C$ is in $C$. We claim that there is a 
sink class $A_i$ among the equivalence classes $A_1,A_2,\dots,A_n$.
To prove this, assume the contrary, i.e.,  
that there is no sink class. Pick an arbitrary
equivalence class $A_{i_1}$. It is not a sink, so we can reach some 
$A_{i_2}$ from it with $i_2\ne i_1$, 
and we can reach $A_{i_3}$ from $A_{i_2}$, with $i_2\ne i_3$, and so on. 
For some $k<l$, we must have $i_k=i_{l+1}$, which is a contradiction, since 
$A_{i_k}\ne A_{i_{k+1}}$, and yet they are reachable from each other.  

Now we continue the proof of the lemma by induction on $|A|=N$. It is
clear that the statement is true when $|A|=1$ . Suppose that the
statement is true when $|A|<N$, and let us prove it for $|A|=N$. 
Let $C$ be a sink class, and
consider $A'=A\setminus C$. By the induction assumption, $\operatorname{d}$
exists on $A'$ and $C$. Since $C$ is a sink, if $y\in A'$ and $z\in C$
then $y$ is not reachable from $z$. So we may multiply the function
$\operatorname{d}$ on $A'$ by a large integer (without changing
$\operatorname{d}$ on $C$) so that $\operatorname{d}(y)\geq
\operatorname{d}(z)M(y,z)$ for all $y\in A',\ z\in C$ such that $z$ is
reachable from $y$. Then $\operatorname{d}$ satisfies the required
condition.
 \end{proof}

\begin{cor}\label{corolld}
For the semigroup $S$, there exists a positive integer 
function $\operatorname{d}$ on $A$ such that 
\begin{equation}\label{ineqd}
\operatorname{d}(y)\geq \operatorname{d}(z)M(y,z)
\end{equation}
whenever $z$ is $y$-persistent. 
\end{cor}

\begin{proof} This follows from Lemma \ref{lem2.3} and Lemma \ref{lem2.4}, applying the latter to the set $P$ of 
pairs $(y,z)$ such that $z$ is $y$-persistent (and $A$, $M$ as above). 
\end{proof} 

Now pick a function $\operatorname{d}$ as  in Corollary \ref{corolld}. 
We extend this function from $A$ to $S$ by setting $\operatorname{d}(x^i)=\operatorname{ d}(x)i$, $i\in \Bbb Z_{>0}$. 
 
 \begin{lemma}
 \label{lem2.5}
 There exists a constant $K>0$ such that $\operatorname{d}(x)\le K$ and $\operatorname{d}(xu)\leq K+\operatorname{d}(u)$ for every $u\in S$ and $x\in A$.
 \end{lemma}
 
 \begin{proof}
 Let $y\in A$. It suffices to prove that there exists $K_y>0$ such that 
 $$
 \operatorname{d}(xy^i)\leq K_y+\operatorname{d}(y^i)=K_y+\operatorname{d}(y)i
 $$
 for all $i$. Then we can take $K'={\rm max}_y K_y$, and then pick $K\ge K'$ so that $\operatorname{d}(x)\le K$ for all $x\in A$. Further, it is enough to show that for each $(x,y)$-persistent
 $z$, there exists $K_{y,z}>0$ such that $\operatorname{d}(xy^i)\leq K_{y,z}+\operatorname{d}(y)i$ whenever $f(x,y,i)=z$. Indeed, 
 taking $K_y':={\rm max}_z K_{y,z}$, we see that the inequality $\operatorname{d}(xy^i)\leq K_y'+\operatorname{d}(y)i$ holds for almost all $i$
 (namely, for all $i$ such that $f(x,y,i)$ is $(x,y)$-persistent), so there exists $K_y\ge K_y'$ such that 
 $\operatorname{d}(xy^i)\leq K_y+\operatorname{d}(y)i$ for all $i$. 
 
 By Lemma \ref{firstlem}, if $f(x,y,i)=z$ then $xy^i=z^{n_0+M(y,z)(i-i_0)}$,  
 so by \eqref{ineqd}, 
 $$
\operatorname{d}(xy^i)=\operatorname{d}(z^{n_0+M(y,z)(i-i_0)})=\operatorname{d}(z)(n_0+M(y,z)(i-i_0))\le 
$$
$$
\le \operatorname{d}(z)n_0+\operatorname{d}(y)(i-i_0)=\operatorname{d}(z)n_0-\operatorname{d}(y)i_0+\operatorname{d}(y)i.
 $$
 So we may take $K_{y,z}$ to be any positive number such that 
 $K_{y,z}\ge \operatorname{d}(z)n_0-\operatorname{d}(y)i_0$.
 \end{proof} 
 
 \begin{cor}
 \label{cor2.6}
 If $u$ is a word of length $n$ then $\operatorname{d}(u)\leq Kn.$
\end{cor}

\begin{proof}
This clearly follows from Lemma \ref{lem2.5} by induction in $n$. 
\end{proof} 
 
 \begin{prop}
 \label{prop2.7}
 For $r>0$, let $I(r)=\{w\in S\ \text{such that}\ \operatorname{d}(w)\leq r\}$. 
 Then there exists $L>0$ such that $|I(r)|\leq Lr$ for all $r$. 
 \end{prop} 
 
 \begin{proof}
 Let $a\in A$ and suppose $a^m\in I(r)$. Since $\operatorname{d}(a^m)=\operatorname{d}(a)m$, we get that $m\le \frac{r}{\operatorname{d}(a)}$. 
 So the number of elements of $I(r)$ of the form $a^m$ is $\big[\frac{r}{\operatorname{d}(a)}\big]$. 
 Hence we may take $L=\sum_{a\in A}\frac{1}{\operatorname{d}(a)}$.
  \end{proof}
  
 \begin{thm}\label{lingrowth}
 \label{thm2.8}
 The semigroup $S$ has linear growth.
 \end{thm}
 \begin{proof}
 Let $J(m)$ be the set of $w\in S$ which can be represented by a word of length $\le m$. 
 Then by Corollary \ref{cor2.6} $J(m)\subset I(Km)$. 
 Hence $|J(m)|\leq |I(Km)|\le LKm$ by Proposition \ref{prop2.7}.
 This implies the theorem. 
  \end{proof}
  
\begin{cor} Let $\overline{S}:=1\cup S$ be the monoid obtained by adding a unit to $S$, and $A=\Bbb F\overline{S}$ be the algebra spanned by $\overline{S}$ 
over any field $\Bbb F$. Then $A$ is a PI algebra.  
\end{cor}

\begin{proof} By a theorem of Small, Stafford, and Warfield \cite{SSW}
any finitely generated algebra of linear growth (i.e., GK dimension 1) is PI. 
Thus the result follows from Theorem \ref{lingrowth}.   
\end{proof}

\begin{flushleft}
Mathematical Sciences Department\\
Princess Nourah bint Abdulrahman University\\
  P.O.Box 84428, Riyadh 11671\\

\smallskip
\text{E-mail address}: \texttt{nhabughazala@pnu.edu.sa}

\end{flushleft}

\begin{flushleft}
Department of Mathematics\\
Massachusetts Institute of Technology\\
 Cambridge, MA, 02139\\

\smallskip
\text{E-mail address}: \texttt{etingof@math.mit.edu}

\end{flushleft}


\begin{thebibliography}{999}

\bibitem{abughazalah13}
N. Abughazalah,  N. Ruskuc, "On disjoint unions of finitely many copies of the free monogenic semigroup." Semigroup Forum. Vol. 87. No. 1. Springer US, 2013.. 

\bibitem{abughazalah14}
N. Abughazalah, J. D Mitchill, Y. Peresse, N. Ruskuc, A classification of disjoint unions of two or three free monogenic semigroup, Semigroup Forum, Online first December (2014).

\bibitem{araujo01}
I.M. Ara\'{u}jo, M.J.J. Branco, V.H. Fernandes, G.M.S. Gomes, and N. Ru\v{s}kuc, On generators
and relations for unions of semigroups, Semigroup Forum 63 (2001), 49?-62. 


\bibitem{golubov75}
E.A. Golubov. Finitely approximable regular semigroups (Russian), Mat.
Zametki 17 (1975) 423--432. (English translation: Math. Notes 17 (1975) 247--251.)


\bibitem{grillet95}
P.A. Grillet, Semigroups, Marcel Dekker, New York, 1995.

\bibitem{higgins1972}
J. Higgins, Subsemigroups of the additive positive integers, Fibonacci Quart, 10:225"1¤70, 1972.

\bibitem{howie95}
J.M. Howie, Fundamentals of Semigroup Theory, Clarendon Press, Oxford, 1995.

\bibitem{SSW}
L. Small, J. Stafford, R. Warfield. Affine algebras of Gelfand-Kirillov dimension one are PI. Math. Proc. Cambridge Philos. Soc., 97 (1984), pp. 407"1¤74. 

\end{thebibliography}
\end{document}